\documentclass{amsart}
\usepackage{mathptmx}
\usepackage{amssymb}
\usepackage{amsfonts}
\usepackage{amsmath}
\usepackage{graphicx}
\usepackage{shadow}
\usepackage[all]{xy}
\usepackage{hyperref}
\newtheorem{theorem}{Theorem}[section]
\newtheorem{lemma}[theorem]{Lemma}
\newtheorem{corollary}[theorem]{Corollary}
\newtheorem{proposition}[theorem]{Proposition}

\theoremstyle{definition}
\newtheorem{definition}[theorem]{Definition}

\theoremstyle{example}
\newtheorem{example}[theorem]{Example}

\theoremstyle{remark}
\newtheorem{remark}[theorem]{Remark}

\theoremstyle{clm}

\DeclareMathOperator{\qf}{qf}

\DeclareMathOperator{\Hom}{Hom}

\DeclareMathOperator{\Spec}{Spec}
\DeclareMathOperator{\Max}{Max}

\DeclareMathOperator{\wdim}{w.gl.dim}

\DeclareMathOperator{\Ann}{Ann}
\DeclareMathOperator{\Nil}{Nil}
\DeclareMathOperator{\Ze}{Z}
\DeclareMathOperator{\End}{End}
\DeclareMathOperator{\Gen}{Gen}

\newcommand{\field}[1]{\mathbb{#1}}

\newcommand{\Z}{\field{Z}}

\newcommand{\F}{\field{F}}
\newcommand{\K}{\field{K}}
\newcommand{\M}{\field{M}}
\def\1{{\rm (1)}}
\def\2{{\rm (2)}}
\def\3{{\rm (3)}}
\def\4{{\rm (4)}}
\def\5{{\rm (5)}}
\begin{document}
%%%%%%%%%%%%%%%%%%%%%%%%%%%%%%%%%%%%%%%%%%%%%%%%%%%%%%%%%%%%%%%%%%%%%%%%%%%%%%%%%%%%%%%%%%%%%%%%%%%%%%%%%%%%%%%%%%%%%
%%%%%%%%%%%%%%%%%%%%%%%%%%%%%%%%%%%%%%%%%%%%%%%%%%%%%%%%%%%%%%%%%%%%%%%%%%%%%%%%%%%%%%%%%%%%%%%%%%%%%%%%%%%%%%%%%%%%%
%%%%%%%%%%%%%%%%%%%%%%%%%%%%%%%%%%%%%%%%%%%%%%%%%%%%%%%%%%%%%%%%%%%%%%%%%%%%%%%%%%%%%%%%%%%%%%%%%%%%%%%%%%%%%%%%%%%%%
%%%%%%%%%%%%%%%%%%%%%%%%%%%%%%%%%%%%%%%%%%%%%%%%%%%%%%%%%%%%%%%%%%%%%%%%%%%%%%%%%%%%%%%%%%%%%%%%%%%%%%%%%%%%%%%%%%%%%
\title[Commutative rings in which every finitely generated ideal is quasi-projective]{Commutative rings in which every finitely generated ideal is quasi-projective}

\author{J. Abuhlail}
\address{Department of Mathematics and Statistics, King Fahd University of Petroleum \& Minerals, P. O. Box 5046, Dhahran 31261, KSA}
\email{abuhlail@kfupm.edu.sa}

\author{M. Jarrar}
\address{Department of Mathematics and Statistics, King Fahd University of Petroleum \& Minerals, P. O. Box 5046, Dhahran 31261, KSA}
\email{mojarrar@kfupm.edu.sa}

\author{S. Kabbaj}
\address{Department of Mathematics and Statistics, King Fahd University of Petroleum \& Minerals, P. O. Box 5046, Dhahran 31261, KSA}
\email{kabbaj@kfupm.edu.sa}
\thanks{The 1$^\mathrm{st}$ and 2$^\mathrm{nd}$ authors were supported by KFUPM under Research Grant \# MS/Rings/351.}

\date{\today}
%\dedicatory{}

%%%%%%%%%%%%%%%%%%%%%%%%%%%%%%%%%%%%%%%%%%%%%%%%%%%%%%%%%%%%%%%%%
%%%%%%%%%%%%%%%%%%%%%%%%%%%%%%%%%%%%%%%%%%%%%%%%%%%%%%%%%%%%%%%%%
\begin{abstract}
This paper studies the multiplicative ideal structure of commutative rings in which every finitely generated ideal is quasi-projective. Section 2 provides some preliminaries on quasi-projective modules over commutative rings. Section 3 investigates the correlation with well-known Pr\"ufer conditions; namely, we prove that this class of rings stands strictly between the two classes of arithmetical rings and Gaussian rings. Thereby, we generalize Osofsky's theorem on the weak global dimension of arithmetical rings and partially resolve Bazzoni-Glaz's related conjecture on Gaussian rings. We also establish an analogue of Bazzoni-Glaz results on the transfer of Pr\"ufer conditions between a ring and its total ring of quotients. Section 4 examines various contexts of trivial ring extensions in order to build new and original examples of rings where all finitely generated ideals are subject to quasi-projectivity, marking their distinction from related classes of Pr\"ufer rings.
\end{abstract}
\maketitle

%%%%%%%%%%%%%%%%%%%%%%%%%%%%%%%%%%%%%%%%%%%%%%%%%%%%%%%%%%%%%%%%%%%%%%%%%%%%%%%%%%%%%%%%%%%%%%%%%%%%%%%%%%%%%%%%%%%%%
%%%%%%%%%%%%%%%%%%%%%%%%%%%%%%%%%%%%%%%%%%%%%%%%%%%%%%%%%%%%%%%%%%%%%%%%%%%%%%%%%%%%%%%%%%%%%%%%%%%%%%%%%%%%%%%%%%%%%
%%%%%%%%%%%%%%%%%%%%%%%%%%%%%%%%%%%%%%%%%%%%%%%%%%%%%%%%%%%%%%%%%%%%%%%%%%%%%%%%%%%%%%%%%%%%%%%%%%%%%%%%%%%%%%%%%%%%%
%%%%%%%%%%%%%%%%%%%%%%%%%%%%%%%%%%%%%%%%%%%%%%%%%%%%%%%%%%%%%%%%%%%%%%%%%%%%%%%%%%%%%%%%%%%%%%%%%%%%%%%%%%%%%%%%%%%%%
\section{Introduction}

All rings considered in this paper, unless otherwise specified, are commutative with identity element and all modules are
unital. There are five well-known extensions of the notion of Pr\"ufer domain \cite{K,P} to arbitrary rings (i.e., with zero divisors). Namely, for a ring $R$,
\1  $R$ is semihereditary, i.e., every finitely generated ideal of $R$ is projective \cite{CE};
\2  $R$ has weak global dimension $\leq1$ \cite{G1,G2};
\3 $R$ is arithmetical, i.e., every finitely generated ideal of $R$ is locally principal \cite{Fu,J};
\4 $R$ is Gaussian, i.e., $c(fg)=c(f)c(g)$ for any polynomials $f,g$ with coefficients in $R$, where $c(f)$ denotes the content of $f$ \cite{T};
\5 $R$ is Pr\"ufer, i.e., every finitely generated regular ideal of $R$ is projective \cite{BS,Gr}.

In the domain context, all these forms coincide with the original definition of a Pr\"ufer domain \cite{G3}, that is, every non-zero finitely generated ideal is invertible \cite{P}. Pr\"ufer domains occur naturally in several areas of commutative algebra, including valuation theory, star and semistar operations, dimension theory, representations of overrings, trace properties, in addition to several homological extensions.

In 1970 Koehler \cite{Koe} studied associative rings for which every cyclic module is quasi-projective. She noticed that any commutative ring satisfies this property. Later, rings in which every left ideal is quasi-projective were studied by Jain and others \cite{JS,GJ} and called left qp-rings. Several characterizations of (semi-)perfect qp-rings were obtained. Moreover, Mohammad \cite{Moh} and  Singh-Mohammad \cite{SM} studied local or semi-perfect rings in which every finitely generated ideal is quasi-projective. A ring is said to be an fqp-ring if every finitely generated ideal is quasi-projective.

This paper studies the multiplicative ideal structure of fqp-rings. Section 2 provides details on finitely generated quasi-projective modules over commutative rings (and demonstrates that these coincide with the so-called $\star$-modules). Section 3 investigates the correlation between the fqp-property and well-known Pr\"ufer conditions. In this vein, the first main result (Theorem~\ref{sec:3.1}) asserts that the class of fqp-rings stands strictly between the two classes of arithmetical rings and Gaussian rings; that is, ``\emph{arithmetical ring $\Rightarrow$ fqp-ring $\Rightarrow$ Gaussian ring}."  Further, the second main result (Theorem~\ref{sec:3.4}) extends Osofsky's theorem on the weak global dimension of arithmetical rings and partially resolves Bazzoni-Glaz's related conjecture on Gaussian rings; we prove that ``\emph{the weak global dimension of an fqp-ring is equal to 0, 1, or $\infty$}." The third main result (Theorem~\ref{sec:3.5}) establishes the transfer of the concept of fqp-ring between a local ring and its total ring of quotients; namely, ``\emph{a local ring $R$ is an fqp-ring if and only if $R$ is Pr\"ufer and $Q(R)$ is an fqp-ring}." Section 4 studies the possible transfer of the fqp-property to various contexts of trivial ring extensions. The main result of this section (Theorem~\ref{sec:4.3}) states that ``\emph{if $(A,\frak{m})$ is a local ring, $E$ a nonzero $\frac{A}{\frak{m}}$-vector space, and $R:=A\ltimes~E$ the trivial ring extension of $A$ by $E$, then $R$ is an fqp-ring if and only if $\frak{m}^2=0$.}" This result generates new and original examples of fqp-rings, marking the distinction between the fqp-property and related Pr\"ufer conditions.

The following diagram of implications puts the notion of fqp-ring in perspective within the family of Pr\"ufer-like rings \cite{BG,BG2}, where the third and fourth implications are established by Theorem~\ref{sec:3.1}:\newpage

%\begin{center}
%\fbox{\parbox{6.5cm}{
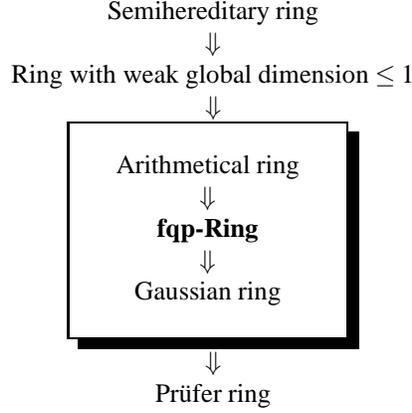
\begin{center}
Semihereditary ring\\
$\Downarrow$\\
Ring with weak global dimension $\leq 1$\\
$\Downarrow$\\
\shabox{\parbox{3cm}{
\begin{center}
Arithmetical ring\\
$\Downarrow$ \\
{\bf fqp-Ring}\\
$\Downarrow$\\
Gaussian ring\\
\end{center}
}}\\
$\Downarrow$\\
Pr\"ufer ring
\end{center}
%}}
%\end{center}

%%%%%%%%%%%%%%%%%%%%%%%%%%%%%%%%%%%%%%%%%%%%%%%%%%%%%%%%%%%%%%%%%%%%%%%%%%%%%%%%%%%%%%%%%%%%%%%%%%%%%%%%%%%%%%%%%%%%%
%%%%%%%%%%%%%%%%%%%%%%%%%%%%%%%%%%%%%%%%%%%%%%%%%%%%%%%%%%%%%%%%%%%%%%%%%%%%%%%%%%%%%%%%%%%%%%%%%%%%%%%%%%%%%%%%%%%%%
%%%%%%%%%%%%%%%%%%%%%%%%%%%%%%%%%%%%%%%%%%%%%%%%%%%%%%%%%%%%%%%%%%%%%%%%%%%%%%%%%%%%%%%%%%%%%%%%%%%%%%%%%%%%%%%%%%%%%
%%%%%%%%%%%%%%%%%%%%%%%%%%%%%%%%%%%%%%%%%%%%%%%%%%%%%%%%%%%%%%%%%%%%%%%%%%%%%%%%%%%%%%%%%%%%%%%%%%%%%%%%%%%%%%%%%%%%%
\section{Preliminaries}\label{sec:2}

This section recalls some preliminaries on the concept of quasi-projective module, including the fact that it coincides with Menini and Orsatti's $\star$-module notion \cite{MO} for finitely generated modules over commutative rings. We give a complete description of quasi-projective modules over arbitrary commutative rings, generalizing Zanardo's description of $\star$-modules over valuation rings \cite{Zan}.

%%%%%%%%%%%%%%%%%%%%%%%%%%%%%%%%%%%%%%%%%%%%%%%%%%%%%%%%%%%%%%%%%%%%%%
%%%%%%%%%%%%%%%%%%%%%%%%%%%%%%%%%%%%%%%%%%%%%%%%%%%%%%%%%%%%%%%%%%%%%%
\begin{definition}\rm
\1 Let $M$ be an $R$-module. An $R$-module $V$ is $M$-projective if the map $\Hom_{R}(V,M)\rightarrow\Hom_{R}(V,M/N)$ is surjective for every submodule $N$ of $M$.\\
\2  $V$ is quasi-projective if $V$ is $V$-projective.
\end{definition}

Let $R$ be a (not necessarily commutative) ring, $\M_{R}$ the category of right $R$-modules, $\M_{S}$ the category of right
$S$-modules, and fix an injective cogenerator $Q_{R}$ in $\M_{R}$. Let $V\in\M_{R}$, $\Ann(V)$ the annihilator of $V$ in $R$, and $V^{\ast}:=%
\Hom_{R}(V,Q)$ considered as a right module over $S:=\mathrm{End}(V)$. Let $\mathrm{Gen}(V)\subseteq \M_{R}$ denote the full
subcategory of $V$-generated right $R$-modules and  $\mathrm{Cogen}(V_{S}^{\ast})\subseteq \M_{S}$ the full subcategory of $V^{\ast}$-cogenerated right $S$-modules. The module $V$ is called a quasi-progenerator if $V$ is quasi-projective and $V$ generates each of its submodules.\\
The fact that  $\Hom_{R}(V,\M_{R})\subseteq \mathrm{Cogen}(V_{S}^{\ast })$
and $\M_{S}\otimes_{S} V\subseteq \mathrm{Gen}(V_{R})$ led Menini and Orsatti in 1989 to introduce and study modules $V_{R}$ for which the two categories $\mathrm{Gen}(V_{R})$ and $\mathrm{Cogen}(V_{S}^{\ast })$ are equivalent \cite{MO}. Several homological characterizations for such modules were given by Colpi
\cite{Col,Col2} who termed them $\star$-modules. Also it is worthwhile recalling that a $\star$-module is necessarily finitely generated (Trlifaj \cite{Tr}). Moreover, in the commutative setting, by combining \cite[Theorem 2.4.5]{CF} \& \cite[Theorem 2.4]{CM} with \cite[18.3 \& 18.5]{Wis} we have:

%%%%%%%%%%%%%%%%%%%%%%%%%%%%%%%%%%%%%%%%%%%%%%%%%%%%%%%%%%%%%%%%%%%%%%
%%%%%%%%%%%%%%%%%%%%%%%%%%%%%%%%%%%%%%%%%%%%%%%%%%%%%%%%%%%%%%%%%%%%%%
\begin{lemma}\label{sg-comm}%this is used in next section
Let $R$ be a commutative ring and $V$ a finitely generated $R$-module. Then the following assertions are equivalent:\\
\1  $V$ is a quasi-progenerator;\\
\2  $V$ is a $\star$-module;\\
\3  $V$ is quasi-projective;\\
\4  $V$ is projective over $\frac{R}{\Ann(V)}$. \qed
\end{lemma}

Next we provide a complete description of quasi-projective modules over arbitrary commutative rings. For the special case of local rings, it recovers the description of $\star$-modules over valuation rings (i.e., chained rings) obtained by Zanardo in \cite{Zan}.

%%%%%%%%%%%%%%%%%%%%%%%%%%%%%%%%%%%%%%%%%%%%%%%%%%%%%%%%%%%%%%%%%%%%%%
%%%%%%%%%%%%%%%%%%%%%%%%%%%%%%%%%%%%%%%%%%%%%%%%%%%%%%%%%%%%%%%%%%%%%%
\begin{theorem}\label{zanardo}
Let $R$ be a commutative ring. A finitely generated $R$-module $V$ is quasi-projective if and
only if $V$ is a direct summand of $(R/I)^{n}$ for some ideal $I$ of $R$ and integer $n\geq 0$. If, moreover, $R$ is local, then $V$ is
quasi-projective if and only if $V\cong (R/I)^{n}$ for some ideal $I$ of $R$ and integer $n\geq 0$.
\end{theorem}

\begin{proof}
Let $V_{R}$ be a finitely generated $R$-module, $J:=\mathrm{Ann}_{R}(V)$, and $\overline{R}:=R/J.$ Assume that $V_{R}$ is quasi-projective. So, $V_{\overline{R}}$ is finitely generated, and projective by Lemma \ref{sg-comm}. It follows that $V_{\overline{R}},$ whence $V_{R},$ is a direct summand of $(R/J)^{n}$ for some $n\geq 0.$ Conversely, let $V$ be a direct summand of $(R/I)^{n}$ for some ideal $I$ of $R$ and integer $n\geq 0$. Then $V_{R/I}$, whence $V_{\overline{R}}$ is finitely generated and projective (notice that $I\subseteq J$). Consequently, $V_{R}$ is a quasi-projective module by Lemma \ref{sg-comm}.\\
Now assume that $R$ is local. If $V_{R}$ is quasi-projective, then $V_{\overline{R}}$ is finitely generated and projective as shown above, whence free since $\overline{R}$ is local. It follows that $V_{R}\cong (R/J)^{n}$ for some $n\geq 0$. The converse was shown to be true for arbitrary commutative rings.
\qed\end{proof}

As a consequence of Theorem~\ref{zanardo}, we generalize Fuller's well-known result on ring extensions \cite[Theorem 2.2]{Ful3} in the commutative context.

%%%%%%%%%%%%%%%%%%%%%%%%%%%%%%%%%%%%%%%%%%%%%%%%%%%%%%%%%%%%%%%%%%%%%%
%%%%%%%%%%%%%%%%%%%%%%%%%%%%%%%%%%%%%%%%%%%%%%%%%%%%%%%%%%%%%%%%%%%%%%
\begin{corollary}\label{fuller}
Let $\xi :A\rightarrow R$ be a morphism of commutative rings. If $U_{A}$ is finitely generated and quasi-projective, then $V_{R}:=R\otimes_{A} U$ is finitely generated and quasi-projective.
\end{corollary}

\begin{proof}
Let $U_{A}$ be a finitely generated quasi-projective $A$-module. Then $U\oplus X=(A/I)^{n}$ for some ideal $I$ of $A$, an integer $n\geq
0$, and an $A$-module $X$. It follows that $(R\otimes _{A}U)\oplus
(R\otimes _{A}X)\cong R\otimes _{A}(A/I)^{n}\cong (R/RI)^{n}$, whence $%
V_{R}:=R\otimes _{A}U$ is finitely generated and quasi-projective by Theorem~\ref{zanardo}.
\qed\end{proof}

Notice that if $U_{A}$ is a $\star $-module, then $U_{A}$ is a
quasi-progenerator and so the faithful module $U_{\overline{A}}$ is projective with $\Gen(U_{\overline{A}}) = \M_{\overline{A}}$, where
 $\overline{A}:=A/\mathrm{Ann}_{A}(U)$. In particular, $U_{\overline{A}}$ generates $V_{\overline{A}}$,
hence $U_{A}$ generates $V_{A}$ (note that $\mathrm{Ann}_{A}(U)\subseteq \mathrm{Ann}%
_{A}(V) $). This shows that the assumption ``$U_{A}$ generates $V_{A}$" in Fuller's result \cite[Theorem 2.2]{Ful3} is automatically
satisfied for $\star $-modules over commutative rings.

\section{Commutative fqp-rings}\label{sec:3}

\begin{definition}\label{sec:2.1}\rm
A commutative ring $R$ is said to be an fqp-ring if every finitely generated ideal of $R$ is quasi-projective.
\end{definition}

In this section we investigate the correlation between (commutative) fqp-rings and the Pr\"ufer-like rings mentioned in the introduction. The first result of this section (Theorem~\ref{sec:3.1}) states that the class of fqp-rings contains strictly the class of arithmetical rings and is contained strictly in the class of Gaussian rings. Its proof provides then specific examples proving that the respective containments are strict. Consequently, fqp-rings stand as a new class of Pr\"ufer-like rings (to the effect that, in the domain context, the fqp-notion coincides with the definition of a Pr\"ufer domain).

In 1969, Osofsky proved that the weak global dimension of an arithmetical ring is either less than or equal to one or infinite \cite{Os}.
Recently, Bazzoni and Glaz studied the homological aspects of Gaussian rings, showing, among others, that Osofsky's result is valid in the context of coherent Gaussian rings (resp., coherent Pr\"ufer rings) \cite[Theorem 3.3]{G2} (resp., \cite[Theorem 6.1]{BG2}). They closed with a conjecture sustaining that ``the weak global dimension of a Gaussian ring is 0, 1, or $\infty$" \cite{BG2}. In this vein, Theorem~\ref{sec:3.4} generalizes Osofsky's theorem as well as validates Bazzoni-Glaz conjecture in the class of fqp-rings.

We close this section with a satisfactory analogue (for fqp-rings) to Bazzoni-Glaz results on the transfer of Pr\"ufer conditions between a ring and its total ring of quotients \cite[Theorems 3.3 \& 3.6 \& 3.7 \& 3.12]{BG2}.

Next we announce the first result of this section.

%%%%%%%%%%%%%%%%%%%%%%%%%%%%%%%%%%%%%%%%%%%%%%%%%%%%%%%%%%%%%%%%%%%%%%
%%%%%%%%%%%%%%%%%%%%%%%%%%%%%%%%%%%%%%%%%%%%%%%%%%%%%%%%%%%%%%%%%%%%%%
\begin{theorem}\label{sec:3.1} For a ring $R$, we have
\begin{center}
$R$ arithmetical $\Rightarrow$ $R$ fqp-ring $\Rightarrow$ $R$ Gaussian
\end{center}
where the implications are irreversible in general.
\end{theorem}

The proof of this theorem involves the following lemmas which are of independent interest.

%%%%%%%%%%%%%%%%%%%%%%%%%%%%%%%%%%%%%%%%%%%%%%%%%%%%%%%%%%%%%%%%%%%%%%
%%%%%%%%%%%%%%%%%%%%%%%%%%%%%%%%%%%%%%%%%%%%%%%%%%%%%%%%%%%%%%%%%%%%%%
\begin{lemma}[{\cite[Lemma 2]{Tug}}]\label{sec:3.1.2}
 Let $R$ be a ring and $M$ a quasi-projective $R$-module. Assume $M=M_1+\ldots+M_n$, where $M_i$ is a
submodule of $M$ for $i=1,\ldots,n$. Then there are endomorphisms $f_i$
of $M$ such that $f_1+\ldots+f_n=1_M$ and $f_i(M)\subseteq M_i$ for $i=1,\ldots,n$.\qed
\end{lemma}

The following result follows directly from Lemma~\ref{sg-comm}.

%%%%%%%%%%%%%%%%%%%%%%%%%%%%%%%%%%%%%%%%%%%%%%%%%%%%%%%%%%%%%%%%%%%%%%
%%%%%%%%%%%%%%%%%%%%%%%%%%%%%%%%%%%%%%%%%%%%%%%%%%%%%%%%%%%%%%%%%%%%%%
\begin{lemma}[\cite{Koe}]\label{sec:3.1.3}
Every cyclic module over a commutative ring is quasi-projective.
\end{lemma}

%%%%%%%%%%%%%%%%%%%%%%%%%%%%%%%%%%%%%%%%%%%%%%%%%%%%%%%%%%%%%%%%%%%%%%
%%%%%%%%%%%%%%%%%%%%%%%%%%%%%%%%%%%%%%%%%%%%%%%%%%%%%%%%%%%%%%%%%%%%%%
\begin{lemma}[{\cite[Corollary 1.2]{FH}}]\label{sec:3.1.4}
Let $\{M_{i}\}_{1\leq i\leq n}$ be a finite family of $R$-modules. Then $\bigoplus_{i=1}^{n}M_{i}$ is quasi-projective if and only if $M_{i}$ is $M_{j}$-projective for all $i, j\in\{1,\ldots, n\}$. \qed
\end{lemma}

%%%%%%%%%%%%%%%%%%%%%%%%%%%%%%%%%%%%%%%%%%%%%%%%%%%%%%%%%%%%%%%%%%%%%%
%%%%%%%%%%%%%%%%%%%%%%%%%%%%%%%%%%%%%%%%%%%%%%%%%%%%%%%%%%%%%%%%%%%%%%
\begin{lemma}\label{sec:3.1.5}
If $R$ is an fqp-ring, then $S^{-1}R$ is an fqp-ring, for any multiplicatively closed subset $S$ of $R$.
\end{lemma}

\begin{proof}
Let $J$ be a finitely  generated ideal of $S^{-1}R$ and let $I$ be a finitely generated ideal of $R$ such that $J:=S^{-1}I$. Then  $I$ is quasi-projective that is faithful over $\frac{R}{\Ann(I)}$. By Lemma~\ref{sg-comm},  $I$ is projective over $\frac{R}{\Ann(I)}$. So that $J:=S^{-1}I$ is projective over $\frac{S^{-1}R}{S^{-1}\Ann(I)}=\frac{S^{-1}R}{\Ann(S^{-1}I)}$. By Lemma~\ref{sg-comm}, $J$ is quasi-projective, as desired.
\qed\end{proof}

%%%%%%%%%%%%%%%%%%%%%%%%%%%%%%%%%%%%%%%%%%%%%%%%%%%%%%%%%%%%%%%%%%%%%%
%%%%%%%%%%%%%%%%%%%%%%%%%%%%%%%%%%%%%%%%%%%%%%%%%%%%%%%%%%%%%%%%%%%%%%
\begin{lemma}[{\cite[19.2]{Wis2} and \cite{Wis3}}] \label{Wisbauer}
Let $R$ be a (commutative) ring and $M$ a finitely generated $R$-module. Then $M$ is quasi-projective if
and only if  $M_{\mathfrak{m}}$ is quasi-projective over $R_{\mathfrak{m}}$ and
$(\End(M))_{\mathfrak{m}}\cong \End_{R_{\mathfrak{m}}}(M_{\mathfrak{m}})$, for every maximal ideal $\mathfrak{m}$ of $R$.\qed
\end{lemma}

%%%%%%%%%%%%%%%%%%%%%%%%%%%%%%%%%%%%%%%%%%%%%%%%%%%%%%%%%%%%%%%%%%%%%%
%%%%%%%%%%%%%%%%%%%%%%%%%%%%%%%%%%%%%%%%%%%%%%%%%%%%%%%%%%%%%%%%%%%%%%
\begin{lemma}\label{sec:3.1.6}
Let $R$ be a local ring and $a,b$ two nonzero elements of $R$ such that $(a)$ and $(b)$ are incomparable. If $(a,b)$ is quasi-projective (in particular, if $R$ is an fqp-ring), then:\\
\1  $(a)\cap (b)=0$,\\
\2  $a^{2}=b^{2}=ab=0$,\\
\3 $\Ann(a)=\Ann(b)$.
\end{lemma}

\begin{proof}
\1 $I:=(a,b)$ is quasi-projective, so by Lemma~\ref{sec:3.1.2}, there exist $f_1,f_2$ in
End$_R$(I) with $f_1(I)\subseteq (a)$, $f_2(I)\subseteq (b)$, and
$f_1+f_2=1_I$. So $$a=f_1(a)+f_2(a)\ ;\ b=f_1(b)+f_2(b).$$
 Let $f_1(a)=x_1a$, $f_2(a)=y_1b$, $f_1(b)=x_2a$, and $f_2(b)=y_2b$. We obtain
 $$a=x_1a+y_1b\ ;\ b=x_2a+y_2b.$$
This forces $x_1$ to be a unit and $1-y_2$ to not be a unit. Let $z\in (a)\cap (b)$; say, $z=c_1a=c_2b$ for some $c_1, c_2\in R$. We get $$z=f_1(c_1a)+f_2(c_2b)=x_1z+y_2z.$$
Therefore $(x_1-(1-y_2))z=0$, hence $z=0$ (since $x_1-(1-y_2)$ is necessarily a unit), as desired.

\2  We have $I=(a)\oplus (b)$. So $(a)$ is $(b)$-projective by Lemma~\ref{sec:3.1.4}. Consider the following diagram of $R$-maps
\begin{equation*}
\xymatrix{ & & (a) \ar^(.4){g}[d] \ar@{.>}[lld]_{f} & & \\
(b)\ar_(.4){\varphi}[rr] & & \dfrac{(b)}{b\mathrm{Ann}(a)} \ar[r] & 0 & }
\end{equation*}
%\[
%\begin{array}{ccc}
%            &                                           &(a)\\
%            &                                           &\downarrow g\\
%(b)         &\buildrel\varphi\over\longrightarrow       &\frac{(b)}{bAnn(a)}
%\end{array}
%\]
where $\varphi$ denotes the canonical map and $g$ is (well) defined by $g(ra)=\overline{rb}$. Since $(a)$ is $(b)$-projective, then there exists an $R$-map  $f : (a)\rightarrow(b)$ with $\varphi\circ f=g$. In particular, $\overline{f(a)}=\overline{b}$ (mod $\frac{(b)}{bAnn(a)}$). Therefore $f(a)=cb$ for some $c\in R$ and hence $cb-b=bd$ for some $d\in \Ann(a)$. Further, since $ab=0$ (recall $(a)\cap (b)=0$), we have $0=f(ab)=bf(a)=cb^2$. Multiplying the above equality by $b$, we get $(d+1)b^2=0$. It follows that $b^2=0$ as $d+1$ is a unit (since $d$ is a zero-divisor and $R$ is local). Likewise, $a^2=0$. Thus $I^2=0$, as claimed.

\3 The above equality $cb-b=bd$ yields $(d+1-c)b=0$. Hence the fact that $d+1$ is a unit forces $c$ to be a unit too (since $b\not=0$). Now, let $x\in \Ann(a).$ Then $0=f(xa)=xf(a)=cxb$, whence $x\in \Ann(b)$. So $\Ann(a)\subseteq \Ann(b)$. Likewise, $\Ann(b)\subseteq \Ann(a)$, completing the proof of the lemma. \qed\end{proof}

It is worthwhile noting that Lemma~\ref{sec:3.1.6} sharpens and recovers \cite[Lemma 3]{Moh} and \cite[Lemma 3]{SM} where the authors require the hypothesis that ``every finitely generated ideal is quasi-projective" (i.e., $R$ is an fqp-ring).

%%%%%%%%%%%%%%%%%%%%%%%%%%%%%%%%%%%%%%%%%%%%%%%%%%%%%%%%%%%%%%%%%%%%%%%%%%%%%%%%%%%%%%%%%%%%%%%%%%%%%%%%%%%%%%%%%%%%%
%%%%%%%%%%%%%%%%%%%%%%%%%%%%%%%%%%%%%%%%%%%%%%%%%%%%%%%%%%%%%%%%%%%%%%%%%%%%%%%%%%%%%%%%%%%%%%%%%%%%%%%%%%%%%%%%%%%%%
\begin{proof}{\it of Theorem~\ref{sec:3.1}}
Assume $R$ to be an arithmetical ring. Let $I$ be a nonzero finitely generated ideal of $R$ and $J$ a subideal of $I$ (possibly equal to 0). Let $P$ be any prime ideal of $R$. Then $I_P:=IR_P$ is a principal  ideal of $R_P$ (possibly equal to $R_P$) and hence quasi-projective by Lemma~\ref{sec:3.1.3}. Moreover, we claim that $$(\Hom_R(I,I))_P\cong\Hom_{R_{P}}(I_P,I_P).$$ We only need to prove $$(\Hom_R(I,I))_P\cong \Hom_{R}(I,I_P).$$ Consider the function
\[
\begin{array}{lrll}

\phi:   &(\Hom_R(I,I))_P      &\longrightarrow    &\Hom_{R}(I,I_P)\\
        &\frac{f}{s}                    &\longrightarrow    &\phi(\frac{f}{s}) : I\rightarrow I_P\ ;\ x\mapsto \frac{f(x)}{s}
\end{array}
\]
Obviously, $\phi$ is a well-defined $R$-map. Moreover, one can easily check that $\phi$ is injective since $I$ is finitely generated. It remains to prove the surjection. Let $g\in \Hom_{R}(I,I_P)$. Clearly, the $R_P$-module $I_P$ is cyclic and so $I_P=aR_P$ for some $a\in I$. Therefore there exists $\lambda\in R$ and $s\in R\setminus P$ such that $g(a)=\frac{\lambda{a}}{s}$. Let $f : I\rightarrow I$ defined by $f(x)=\lambda{x}$. Then $f\in \Hom_R(I,I)$. Let $x\in I$. Further $\frac{x}{1}=\frac{ra}{u}$ for some $r\in R$ and $u\in R\setminus P$, whence $tux=tra$ for some $t\in R\setminus P$. We have
$$\phi(\frac{f}{s})(x)  = \frac{f(x)}{s}
                        =  \frac{\lambda}{s}\frac{{x}}{1}
                        =  \frac{\lambda}{s}\frac{{ra}}{u}
                        = \frac{r}{u}g(a)
                        =  \frac{1}{tu}g(tra)
                        =  \frac{1}{tu}g(tux)
                        =  g(x).$$
This proves the claim. By Lemma~\ref{Wisbauer}, $I$ is quasi-projective and hence $R$ is an fqp-ring, proving the first implication.
Next assume $R$ to be an fqp-ring. The Gaussian notion is a local property, that is, $R$ is Gaussian if and only if $R_{P}$ is Gaussian for every $P\in\Spec(R)$ \cite{BG2}. This fact combined with Lemma~\ref{sec:3.1.5} reduces the proof to the local case. Now, $R$ is a local fqp-ring. Let $a, b$ be any two elements of $R$. We envisage two cases. \emph{Case~1}: Suppose $(a, b)=(a)$ or $(b)$, say, $(a)$. Then $(a, b)^2=(a^2)$ and if in addition $ab=0$, then $ b\in (a)$ implies that $b^2=0$. \emph{Case~2}: Suppose  $I:=(a,b)$ with $I\neq(a)$ and $I\neq(b)$. Obviously, $a\not=0$ and $b\not=0$. By Lemma~\ref{sec:3.1.6}, $I^2=0$. Consequently, both cases satisfy the conditions of \cite[Theorem 2.2(d)]{BG2} and thus $R$ is a  Gaussian ring, proving the second implication.

It remains to show that both implications are, in general, irreversible. This is handled by the next two examples.
\qed\end{proof}

%%%%%%%%%%%%%%%%%%%%%%%%%%%%%%%%%%%%%%%%%%%%%%%%%%%%%%%%%%%%%%%%%%%%%%
%%%%%%%%%%%%%%%%%%%%%%%%%%%%%%%%%%%%%%%%%%%%%%%%%%%%%%%%%%%%%%%%%%%%%%
\begin{example}\label{sec:3.2}\rm
There is an example of an fqp-ring that is not arithmetical.
\end{example}

\begin{proof}
From \cite{G3}, consider the local ring $R:=\frac{\F_{2}[x,y]}{(x,y)^2}\cong\F_{2}[\overline{x},\overline{y}]$ with maximal ideal $\mathfrak{m}:=(\overline{x},\overline{y})$. The proper ideals of $R$ are exactly $(0)$, $(\overline{x})$, $(\overline{y})$, $(\overline{x}+\overline{y})$, and $\mathfrak{m}$. By Lemma~\ref{sec:3.1.3}, $(\overline{x})$, $(\overline{y})$, and $(\overline{x}+\overline{y})$ are quasi-projective. Further $\mathfrak{m}:=(\overline{x})\oplus(\overline{y})$ implies that $\mathfrak{m}$ is quasi-projective by Lemma~\ref{sec:3.1.4}. Hence $R$ is an fqp-ring. Clearly, $R$ is not an arithmetical ring since $\mathfrak{m}$ is not principal.
\qed\end{proof}

%%%%%%%%%%%%%%%%%%%%%%%%%%%%%%%%%%%%%%%%%%%%%%%%%%%%%%%%%%%%%%%%%%%%%%
%%%%%%%%%%%%%%%%%%%%%%%%%%%%%%%%%%%%%%%%%%%%%%%%%%%%%%%%%%%%%%%%%%%%%%
\begin{example}\label{sec:3.3}\rm
There is an example of a Gaussian ring that is not an fqp-ring.
\end{example}

\begin{proof}
Let $\K$ be a field and consider the local Noetherian ring $R:=\frac{\K[x,y]}{(x^2,,xy,y^3)}\cong\K[\overline{x},\overline{y}]$ with maximal ideal $\mathfrak{m}:=(\overline{x},\overline{y})$. One can easily verify that $\Ann(\mathfrak{m})=(\overline{x},\overline{y}^{2})$ and then $\frac{R}{\Ann(\mathfrak{m})}\cong \frac{\K[y]}{(y^2)}$. Therefore $\frac{R}{\Ann(\mathfrak{m})}$ is a principal and hence an arithmetical ring. It follows that $R$ is a Gaussian ring (see first paragraph right after Theorem 2.2 in \cite{BG2}). Finally, we claim that $\mathfrak{m}$ is not quasi-projective. Deny. Since $\mathfrak{m}=(\overline{x},\overline{y})$ with $\mathfrak{m}\neq(\overline{x})$ and $\mathfrak{m}\neq(\overline{y})$, then Lemma~\ref{sec:3.1.6} yields $\mathfrak{m}^2=0$, absurd. Thus $R$ is a not an fqp-ring, as desired.
\qed\end{proof}

Next, in view of Theorem~\ref{sec:3.1} and Example~\ref{sec:3.2}, we extend Osofsky's theorem on the weak global dimension of arithmetical rings to the class of fqp-rings.

%%%%%%%%%%%%%%%%%%%%%%%%%%%%%%%%%%%%%%%%%%%%%%%%%%%%%%%%%%%%%%%%%%%%%%
%%%%%%%%%%%%%%%%%%%%%%%%%%%%%%%%%%%%%%%%%%%%%%%%%%%%%%%%%%%%%%%%%%%%%%
\begin{theorem}\label{sec:3.4}
The weak global dimension of an fqp-ring is equal to 0, 1, or $\infty$.
\end{theorem}

The proof uses the following result.

%%%%%%%%%%%%%%%%%%%%%%%%%%%%%%%%%%%%%%%%%%%%%%%%%%%%%%%%%%%%%%%%%%%%%%
%%%%%%%%%%%%%%%%%%%%%%%%%%%%%%%%%%%%%%%%%%%%%%%%%%%%%%%%%%%%%%%%%%%%%%
\begin{lemma}[{\cite[Theorem 2]{SM}}]\label{sec:3.4.0}
Let $R$ be a local fqp-ring. Then either $\Nil(R)^{2}=0$ or $R$ is a chained ring {\rm (}i.e., its ideals are linearly ordered with respect to inclusion{\rm )}.\qed
\end{lemma}

%%%%%%%%%%%%%%%%%%%%%%%%%%%%%%%%%%%%%%%%%%%%%%%%%%%%%%%%%%%%%%%%%%%%%%%%%%%%%%%%%%%%%%%%%%%%%%%%%%%%%%%%%%%%%%%%%%%%%
%%%%%%%%%%%%%%%%%%%%%%%%%%%%%%%%%%%%%%%%%%%%%%%%%%%%%%%%%%%%%%%%%%%%%%%%%%%%%%%%%%%%%%%%%%%%%%%%%%%%%%%%%%%%%%%%%%%%%
\begin{proof}{\it of Theorem~\ref{sec:3.4}}
Since $\wdim(R)=\sup\{\wdim(R_{\mathfrak{m}})\mid \mathfrak{m}\in\Max(R)\}$, we only need to prove the theorem for the local case. Let $R$ be a local fqp-ring. We envisage two cases. \emph{Case~1}: Suppose $R$ is reduced. Then Theorem~\ref{sec:3.1} combined with \cite[Theorem 2.2]{G2} forces the weak global dimension of $R$ to be less than or equal to one, as desired.  \emph{Case~2}: Suppose $R$ is not reduced. By Lemma~\ref{sec:3.4.0}, $(\Nil(R))^{2}=0$ or $R$ is a chained ring. By Theorem~\ref{sec:3.1}, $R$ is Gaussian, so that the statement  ``$(\Nil(R))^{2}=0$" yields $\wdim(R)=\infty$ by \cite[Theorem 6.4]{BG2}. On the other hand, the statement ``$R$ is a chained ring" implies that $R$ is a local arithmetical ring with zero-divisors (since $\Nil(R)\not=0$), hence $R$ has an infinite weak global dimension (Osofsky \cite{Os}), completing the proof of the theorem.
\qed\end{proof}

In 2005, Glaz proved that Osofsky's result is valid in the class of coherent Gaussian rings \cite[Theorem 3.3]{G2}. In 2007, Bazzoni and Glaz conjectured that the same must hold in the whole class of Gaussian rings \cite{BG2}. Theorem~\ref{sec:3.4} widens the scope of validity of this conjecture beyond the class of coherent Gaussian rings, as shown by next example:

%%%%%%%%%%%%%%%%%%%%%%%%%%%%%%%%%%%%%%%%%%%%%%%%%%%%%%%%%%%%%%%%%%%%%%
%%%%%%%%%%%%%%%%%%%%%%%%%%%%%%%%%%%%%%%%%%%%%%%%%%%%%%%%%%%%%%%%%%%%%%
\begin{example}\label{sec:3.4.1}\rm
There is an example of an fqp-ring that is neither arithmetical nor coherent.
\end{example}

\begin{proof}
Let $\K$ be field and $\{x_{1}, x_{2}, . . . \}$ an infinite set of indeterminates over $\K$. Let $R:=\frac{\K[x_{1}, x_{2}, . . . ]}{\mathfrak{m}^{2}} =\K[\overline{x_{1}}, \overline{x_{2}}, . . . ]$, where $\mathfrak{m}:=(x_{1}, x_{2}, . . . )$. One can easily check that $R$ has the following features:
\begin{enumerate}
\item $R=\K+\frac{\mathfrak{m}}{\mathfrak{m}^{2}}$ is local with maximal ideal $\frac{\mathfrak{m}}{\mathfrak{m}^{2}}$.

\item $\forall\ f\in \frac{\mathfrak{m}}{\mathfrak{m}^{2}}$, $\Ann(f)= \frac{\mathfrak{m}}{\mathfrak{m}^{2}}$.

\item $\forall\ i\not=j$, $(\overline{x_{i}})\cap (\overline{x_{j}})=0$.

\item $\forall\ f\in \frac{\mathfrak{m}}{\mathfrak{m}^{2}}$ and $\forall\ i\geq1$, $(f)\cong (\overline{x_{i}})$.

\item For every finitely generated ideal $I$ of $R$, we have $I\cong\bigoplus_{k=1}^{n}(\overline{x_{i_{k}}})$ for some indeterminates $x_{i_{1}}, . . ., x_{i_{n}}$ in  $\{x_{1}, x_{2}, . . . \}$.
\end{enumerate}

Let $I$ be a finitely generated ideal of $R$. By (4), $(\overline{x_{i}})$ is $(\overline{x_{j}})$-projective for all $i, j\geq1$. So (5) forces $I$ to be quasi-projective by Lemma~\ref{sec:3.1.4}. Therefore $R$ is an fqp-ring. Moreover, by (2), the following sequence of natural homomorphisms
$$0 \rightarrow \frac{\mathfrak{m}}{\mathfrak{m}^{2}}\rightarrow R\rightarrow  R\overline{x_{1}} \rightarrow 0$$
is exact. So $R\overline{x_{1}}$ is not finitely presented and hence $R$ is not coherent. Finally, observe that $R\overline{x_{1}}$ and $R\overline{x_{2}}$ are incomparable so that $R$ is not a chained ring and, hence, not an arithmetical ring (recall $R$ is local).
\qed\end{proof}

In \cite{BG2}, Bazzoni and Glaz proved that a Pr\"ufer ring $R$ satisfies any of the other four Pr\"ufer conditions (mentioned in the introduction) if and only if its total ring of quotients $Q(R)$ satisfies that same condition. This fact narrows the scope of study of the Pr\"ufer conditions to the class of total rings of quotients; specifically, ``a Pr\"ufer ring is Gaussian (resp., is arithmetical, has $\wdim(R)\leq1$, is semihereditary) if and only if so is $Q(R)$" \cite[Theorems 3.3 \& 3.6 \& 3.7 \& 3.12]{BG2}. Next, we establish an analogue for the fqp-property in the local case.

%%%%%%%%%%%%%%%%%%%%%%%%%%%%%%%%%%%%%%%%%%%%%%%%%%%%%%%%%%%%%%%%%%%%%%
%%%%%%%%%%%%%%%%%%%%%%%%%%%%%%%%%%%%%%%%%%%%%%%%%%%%%%%%%%%%%%%%%%%%%%
\begin{theorem}\label{sec:3.5}
Let $R$ be a local ring. Then $R$ is Pr\"ufer and $Q(R)$ is an fqp-ring if and only if $R$ is an fqp-ring.
\end{theorem}

%%%%%%%%%%%%%%%%%%%%%%%%%%%%%%%%%%%%%%%%%%%%%%%%%%%%%%%%%%%%%%%%%%%%%%%%%%%%%%%%%%%%%%%%%%%%%%%%%%%%%%%%%%%%%%%%%%%%%
%%%%%%%%%%%%%%%%%%%%%%%%%%%%%%%%%%%%%%%%%%%%%%%%%%%%%%%%%%%%%%%%%%%%%%%%%%%%%%%%%%%%%%%%%%%%%%%%%%%%%%%%%%%%%%%%%%%%%
\begin{proof}
A Gaussian ring is Pr\"ufer \cite[Theorem 3.4.1]{G3} and \cite[Theorem 6]{Lu}. So in view of Theorem~\ref{sec:3.1} and Lemma~\ref{sec:3.1.5} only the necessity has to be proved. Assume $R$ is Pr\"ufer and $Q(R)$ is an fqp-ring. Notice first that $R$ is a (local) Gaussian ring by
\cite[Theorem 3.3]{BG2} and hence the lattice of its prime ideals is linearly ordered \cite{T}. Therefore the set of zero-divisors $\Ze(R)$ of $R$ is a prime ideal and hence $Q(R)=R_{\Ze(R)}$ is local. Next, let $S$ denote the set of all regular elements of $R$ and let $I$ be a finitely generated ideal of $R$ with a minimal generating set $\{x_1, \ldots, x_n\}$. If $I$ is regular, then $I$ is projective (since $R$ is Pr\"{u}fer). Suppose $I$ is not regular, that is, $I\cap S=\emptyset$. We wish to show that $I$ is quasi-projective. We first claim that
$$\big(\frac{x_{i}}{1}\big)Q(R)\cap \big(\frac{x_{j}}{1}\big)Q(R)=0,\ \forall\ i\not=j\in\{1, \ldots, n\}.$$
Indeed, for any $i\neq j$, the ideals $(\frac{x_{i}}{1})$ and $(\frac{x_{j}}{1})$ are incomparable in $Q(R)$: Otherwise if, say, $\frac{x_{i}}{1}\in (\frac{x_{j}}{1})$, then $sx_{i}= ax_{j}$ for some $a\in R$ and $s\in S$. Since $s$ is regular, the ideal $(a,s)$ is projective in $R$ (which is Pr\"{u}fer). Moreover, by Lemma~\ref{sec:3.1.6}, we obtain $(a,s)=(s)$ or $(a,s)=(a)$ and, in this case, necessarily $a\in S$. It follows that $x_{i}$ and
$x_{j}$ are linearly dependent which contradict minimality. Therefore, by Lemma~\ref{sec:3.1.6} applied to the ideal $(\frac{x_{i}}{1}, \frac{x_{j}}{1})$ in the local fqp-ring $Q(R)$, we get $(\frac{x_{i}}{1})\cap (\frac{x_{j}}{1})=0$, proving the claim. Since $S$ consists of regular elements, then $x_{i}R\cap x_{j}R=0$, for each $i\not=j$, whence $I=\bigoplus_{i=1}^{n} x_{i}R$. Further, by Lemma~\ref{sec:3.1.6}, we have $$\Ann_{Q(R)}\big(\dfrac{x_{i}}{1}\big)=\Ann_{Q(R)}\big(\frac{x_{j}}{1}\big),\ \forall\ i\not=j\in\{1, \ldots, n\}.$$
Therefore, we obtain $$\Ann(x_{i})=\Ann(x_{j}),\ \forall\ i\not=j\in\{1, \ldots, n\}.$$
Consequently, $x_{i}R\cong x_{j}R$ and hence $x_{i}R$ is $x_{j}R$-projective for all $i, j$. Once again, we appeal to  Lemma~\ref{sec:3.1.4} to conclude that $I$ is quasi-projective, as desired.
\qed\end{proof}

The global case holds for coherent rings as shown next.

%%%%%%%%%%%%%%%%%%%%%%%%%%%%%%%%%%%%%%%%%%%%%%%%%%%%%%%%%%%%%%%%%%%%%%
%%%%%%%%%%%%%%%%%%%%%%%%%%%%%%%%%%%%%%%%%%%%%%%%%%%%%%%%%%%%%%%%%%%%%%
\begin{corollary}\label{sec:3.6}
Let $R$ be a coherent ring. Then $R$ is Pr\"ufer and $Q(R)$ is an fqp-ring if and only if $R$ is an fqp-ring.
\end{corollary}

\begin{proof}
Here too only necessity has to be proved. Assume $R$ is Pr\"ufer and $Q(R)$ is an fqp-ring and let $I$ be a finitely generated ideal of $R$.
By \cite[Theorem 3.3]{BG2}, $R$ is Gaussian. Let $P$ be a prime ideal of $R$. Then $R_{P}$ is a local Pr\"ufer ring (since Gaussian). Moreover, by \cite[Theorem 3.4]{BG2}, the total ring of quotients of $R_{P}$ is a localization of Q(R) (with respect to a multiplicative subset of R) and hence an fqp-ring by Lemma~\ref{sec:3.1.5}. By Theorem~\ref{sec:3.5}, $R_{P}$ is an fqp-ring. Consequently, $I$ is locally quasi-projective. Since $I$ is finitely presented, then $I$ is quasi-projective \cite[Theorem 2]{Fat}, as desired.
\qed\end{proof}

We close this section with a discussion of the global case. Recall first that the Gaussian and arithmetical properties are local, i.e.,
$R$ is Gaussian (resp., arithmetical) if and only if $R_{\mathfrak{m}}$ is Gaussian (resp., arithmetical) for every maximal ideal $\mathfrak{m}$ of $R$. The same holds for rings with weak global dimension $\leq1$. We were not able to prove or disprove this fact for fqp-rings. Moreover, the transfer result \cite[Theorem 3.12(i)]{BG2} for the semihereditary notion (which is not a local property) was made possible by Endo's result that ``a total ring of quotients is semihereditary if and only if it is von Neumann regular" \cite{E}. No similar phenomenon occurs for the fqp-property; namely, a total ring of quotients that is an fqp-ring is not necessarily arithmetical (see Example~\ref{sec:3.2}). Based on the above discussion, one wonders if Theorem~\ref{sec:3.5} is true for all rings. We have not succeeded to prove this fact.

%%%%%%%%%%%%%%%%%%%%%%%%%%%%%%%%%%%%%%%%%%%%%%%%%%%%%%%%%%%%%%%%%%%%%%%%%%%%%%%%%%%%%%%%%%%%%%%%%%%%%%%%%%%%%%%%%%%%%
%%%%%%%%%%%%%%%%%%%%%%%%%%%%%%%%%%%%%%%%%%%%%%%%%%%%%%%%%%%%%%%%%%%%%%%%%%%%%%%%%%%%%%%%%%%%%%%%%%%%%%%%%%%%%%%%%%%%%
%%%%%%%%%%%%%%%%%%%%%%%%%%%%%%%%%%%%%%%%%%%%%%%%%%%%%%%%%%%%%%%%%%%%%%%%%%%%%%%%%%%%%%%%%%%%%%%%%%%%%%%%%%%%%%%%%%%%%
%%%%%%%%%%%%%%%%%%%%%%%%%%%%%%%%%%%%%%%%%%%%%%%%%%%%%%%%%%%%%%%%%%%%%%%%%%%%%%%%%%%%%%%%%%%%%%%%%%%%%%%%%%%%%%%%%%%%%
\section{Examples via trivial ring extensions}\label{sec:4}

This section studies the fqp-property in various trivial ring extensions. Our objective is to generate new
and original examples to enrich the current literature with new families of fqp-rings. It is worthwhile noticing that trivial extensions have been thoroughly investigated in \cite{BKM} for the other five Pr\"ufer conditions (mentioned in the introduction).

Let $A$ be a ring and $E$ an $A$-module. The trivial (ring) extension of $A$ by $E$ (also called the idealization
of $E$ over $A$) is the ring $R:= A\ltimes~E$ whose underlying group is $A \times E$ with multiplication given by $(a_{1}, e_{1})(a_{2}, e_{2}) =
(a_{1}a_{2}, a_{1}e_{2}+a_{2}e_{1})$. For the reader's convenience, recall that if $I$ is an ideal of $A$ and $E'$ is a submodule of $E$ such that $IE \subseteq E'$, then $J := I \ltimes E'$ is an ideal of $R$; ideals of $R$ need not be of this form \cite[Example 2.5]{KM}. However, prime (resp., maximal) ideals of $R$ have the form  $p\ltimes E$, where $p$ is a prime (resp., maximal) ideal of $A$ \cite[Theorem 25.1(3)]{H}. Also an ideal of $R$ of the form $I\ltimes IE$, where $I$ is an ideal of $A$, is finitely generated if and only if $I$ is finitely generated \cite[p. 141]{G1}. A suitable background on commutative trivial ring extensions is \cite{G1,H}.

We first state a necessary condition for the inheritance of the fqp-property in a general context of trivial extensions.

%%%%%%%%%%%%%%%%%%%%%%%%%%%%%%%%%%%%%%%%%%%%%%%%%%%%%%%%%%%%%%%%%%%%%%
%%%%%%%%%%%%%%%%%%%%%%%%%%%%%%%%%%%%%%%%%%%%%%%%%%%%%%%%%%%%%%%%%%%%%%
\begin{proposition}\label{sec:4.1}
 Let $A$ be a  ring, $E$ an $A$-module, and $R:=A\ltimes E$ the trivial ring extension of $A$ by $E$. If $R$ is an fqp-ring, then so is $A$.
\end{proposition}

\begin{proof}
Assume that $R$ is an fqp-ring. Let $I$ be a finitely generated ideal of $A$, $J$ a subideal of $I$, and $f\in \Hom_{A}(I,I/J)$. We wish to prove the existence of $h\in \Hom_{A}(I,I)$ such that $f(x)=\overline{h(x)}$ (mod $J$), for every $x\in I$. Clearly, $I\ltimes IE$ is a finitely generated ideal of $R$ and $J\ltimes IE$ a subideal of $I\ltimes IE$. Let $F : I\ltimes IE \longrightarrow \dfrac{I\ltimes IE}{J\ltimes IE}$ defined by $F(x,e)=\overline{(a,0)}$ (mod $J\ltimes IE$) where $a\in I$ with $f(x)=\overline{a}$ (mod $J$). It is easily seen that $F$ is a well-defined $R$-map. By assumption, $I\ltimes IE$ is quasi-projective. So there exists $H\in \Hom_{R}(I\ltimes IE,I\ltimes IE)$ such that $F(x,e)=\overline{H(x,e)}$ (mod $J\ltimes IE$), for every $(x,e)\in I\ltimes IE$. Now, for each $x\in I$, let $h(x)$ denote the first coordinate of $H(x,0)$; that is, $H(x,0)=(h(x),e_{x})$ for some $e_{x}\in IE$. One can easily check that $h : I\longrightarrow I$ is an $A$-map. Moreover, let $x\in I$ and $a\in I$ with $f(x)=\overline{a}$. We have $\overline{(a,0)}=F(x,0)=\overline{H(x,0)}=\overline{(h(x),e_{x})}$ (mod $J\ltimes IE$). It follows that $f(x)=\overline{a}=\overline{h(x)}$ (mod $J$), as desired.
\qed\end{proof}

%%%%%%%%%%%%%%%%%%%%%%%%%%%%%%%%%%%%%%%%%%%%%%%%%%%%%%%%%%%%%%%%%%%%%%
%%%%%%%%%%%%%%%%%%%%%%%%%%%%%%%%%%%%%%%%%%%%%%%%%%%%%%%%%%%%%%%%%%%%%%
\begin{remark}\label{sec:4.1.1}\rm
One can also prove Proposition~\ref{sec:4.1} via Corollary~\ref{fuller}. Indeed, assume $R:=A\ltimes E$ is an fqp-ring and let $I$ be a finitely generated ideal of $A$. Then $U_{R}:=I\ltimes IE$ is a finitely generated ideal of $R$ and hence quasi-projective. Now consider the ring homomorphism $\varphi : R\longrightarrow A$ defined by $\varphi(a,e)=a$. Clearly, the fact $A\cong \frac{R}{0\ltimes E}$ leads to the conclusion (to the effect that $A\otimes_{R} U\cong \frac{R}{0\ltimes E}\otimes_{R} I\ltimes IE\cong \frac{I\ltimes IE}{0\ltimes IE}\cong I$).
\end{remark}

Example~\ref{sec:4.5} below provides a counter-example for the converse of Proposition~\ref{sec:4.1}. The next two results establish necessary and sufficient conditions for the transfer of the fqp-property in special contexts of trivial extensions. We first examine the case of trivial extensions of integral domains.

%%%%%%%%%%%%%%%%%%%%%%%%%%%%%%%%%%%%%%%%%%%%%%%%%%%%%%%%%%%%%%%%%%%%%%
%%%%%%%%%%%%%%%%%%%%%%%%%%%%%%%%%%%%%%%%%%%%%%%%%%%%%%%%%%%%%%%%%%%%%%
\begin{theorem}\label{sec:4.2}
Let $A\subseteq B$ be an extension of domains and $K :=\qf(A)$. Let $R:=A \ltimes~B$ be the trivial ring extension of $A$ by $B$. Then the following statements are equivalent:\\
\1  $A$ is a Pr\"ufer domain with $K\subseteq B$;\\
\2  $R$ is a Pr\"ufer ring;\\
\3 $R$ is a Gaussian ring;\\
\4 $R$ is an fqp-ring.
\end{theorem}

\begin{proof}
The implications \1  $\Longleftrightarrow$ \2  $\Longleftrightarrow$ \3 and \4 $\Longrightarrow$ \3 are handled by \cite[Theorem 2.1]{BKM} and Theorem~\ref{sec:3.1}, respectively. It remains to prove \3 $\Longrightarrow$ \4. Notice first that $(a,b)\in R$ is regular if and only if $a\not=0$. Assume that $R$ is Gaussian and let $I$ be a (non-zero) finitely generated ideal of $R$.  If $I$ contains a regular element, then $I$ is projective (since $R$ is a Pr\"ufer ring). If $I\subseteq 0\ltimes B$, then $I$ is a torsion free $A$-module and hence projective (since $A$ is a Pr\"ufer domain). But $A\cong \frac{R}{0\ltimes B}$ with $\rm{Ann}(I)=0\ltimes B$, hence $I$ is quasi-projective by Lemma~\ref{sg-comm}. Therefore $R$ is an fqp-ring.
\qed\end{proof}

Next we examine the case of trivial extensions of local rings by vector spaces over the residue fields.

%%%%%%%%%%%%%%%%%%%%%%%%%%%%%%%%%%%%%%%%%%%%%%%%%%%%%%%%%%%%%%%%%%%%%%
%%%%%%%%%%%%%%%%%%%%%%%%%%%%%%%%%%%%%%%%%%%%%%%%%%%%%%%%%%%%%%%%%%%%%%
\begin{theorem}\label{sec:4.3}
Let $(A,\frak{m})$ be a local ring and $E$ a nonzero $\dfrac{A}{\frak{m}}$-vector space. Let $R:=A\ltimes~E$ be the trivial ring
extension of $A$ by $E$. Then $R$ is an fqp-ring if and only if $\frak{m}^2=0$.
\end{theorem}

The proof lies on the next preliminary results.

%%%%%%%%%%%%%%%%%%%%%%%%%%%%%%%%%%%%%%%%%%%%%%%%%%%%%%%%%%%%%%%%%%%%%%
%%%%%%%%%%%%%%%%%%%%%%%%%%%%%%%%%%%%%%%%%%%%%%%%%%%%%%%%%%%%%%%%%%%%%%
\begin{lemma}\label{sec:4.4}
Let $R$ be a local fqp-ring which is not a chained ring. Then $\Ze(R)=\Nil(R)$.
\end{lemma}

\begin{proof}
Let $s\in \Ze(R)$. Assume by way of contradiction that $s\notin\Nil(R)$. Let $x,y$ be two nonzero elements of $R$ such that $(x)$ and $(y)$ are incomparable (since $R$ is not a chained ring). Lemma~\ref{sec:3.1.6} forces $(x)$ and $(s)$ to be comparable and a fortiori $x\in (s)$. Likewise $y\in (s)$; say, $x=sx'$ and $y=sy'$ for some $x', y'\in R$. Necessarily, $(x')$ and $(y')$ are incomparable and hence $(x')\cap(y')=0$ (by the same lemma). Now let $0\neq t \in R$ such that $st=0$. Next let's consider three cases. If $(x')$ and $(t)$ are incomparable, then $\Ann(x')=\Ann(t)$ by Lemma~\ref{sec:3.1.6}(3). It follows that $x=sx'=0$, absurd. If $(t)\subseteq(x')$, then $(t)\cap(y')\subseteq (x')\cap(y')=0$. So $(y')$ and $(t)$ are incomparable, whence similar arguments yield $y=sy'=0$, absurd. If $(x')\subseteq(t)$; say, $x'=rt$ for some $r\in R$, then $x=sx'=str=0$, absurd. All possible cases end up with an absurdity, the desired contradiction. Therefore $s\in\Nil(R)$ and thus $\Ze(R)=\Nil(R)$.
\qed\end{proof}

%%%%%%%%%%%%%%%%%%%%%%%%%%%%%%%%%%%%%%%%%%%%%%%%%%%%%%%%%%%%%%%%%%%%%%
%%%%%%%%%%%%%%%%%%%%%%%%%%%%%%%%%%%%%%%%%%%%%%%%%%%%%%%%%%%%%%%%%%%%%%
\begin{lemma}\label{sec:4.4.1}
Let $(R,\frak{m})$ be a local ring. If $\frak{m}^{2}=0$, then $R$ is an fqp-ring.
\end{lemma}

\begin{proof}
Let $I$ be a nonzero proper finitely generated ideal of $R$. Then $\Ann(I)=\frak{m}$. So $\dfrac{R}{\Ann(I)}\cong \dfrac{A}{\frak{m}}$. Hence $I$ is a free $\frac{R}{\Ann(I)}$-module, whence $I$ is quasi-projective by Lemma~\ref{sg-comm}. Consequently, $R$ is an fqp-ring.
\qed\end{proof}

%%%%%%%%%%%%%%%%%%%%%%%%%%%%%%%%%%%%%%%%%%%%%%%%%%%%%%%%%%%%%%%%%%%%%%
\begin{proof}{\it of Theorem~\ref{sec:4.3}} Recall first that $R$ is local with maximal ideal $\frak{m}\ltimes~E$ as well as a total ring of quotients (i.e., $Q(R)=R$). Now suppose that $R$ is an fqp-ring. Without loss of generality, we may assume $A$ not to be a field. Notice that $R$ is not a chained ring since, for $e:=(1, 0, 0, \ldots)\in E$ and $0\neq a\in \frak{m}$, \big((a,0)\big) and \big((0,e)\big) are incomparable. Therefore Lemma~\ref{sec:4.4} yields $\frak{m}\ltimes~E=\Ze(R)=\Nil(R)$. By Lemma~\ref{sec:3.4.0}, $(\frak{m}\ltimes~E)^{2}=0$, hence $\frak{m}^{2}=0$, as desired.

Conversely, suppose $\frak{m}^2=0$. Then $(\frak{m}\ltimes E)^2=0$ which leads to the conclusion via Lemma~\ref{sec:4.4.1}, completing the proof of the theorem.
\qed\end{proof}

\cite[Theorem 3.1]{BKM} states that ``$R:=A\ltimes~E$ is Gaussian if and only if so is $A$" and ``$R$ is arithmetical if and only if $A := K$ is a field and $\dim_{K}E = 1$." Theorem~\ref{sec:4.3} generates new and original examples of rings with zero-divisors subject to Pr\"ufer conditions as shown below.

%%%%%%%%%%%%%%%%%%%%%%%%%%%%%%%%%%%%%%%%%%%%%%%%%%%%%%%%%%%%%%%%%%%%%%
%%%%%%%%%%%%%%%%%%%%%%%%%%%%%%%%%%%%%%%%%%%%%%%%%%%%%%%%%%%%%%%%%%%%%%
\begin{example}\label{sec:4.5}
 $R:=\dfrac{\Z}{8\Z}\ltimes~\dfrac{\Z}{2\Z}$ is a Gaussian total ring of quotients which is not an fqp-ring.
\end{example}

%%%%%%%%%%%%%%%%%%%%%%%%%%%%%%%%%%%%%%%%%%%%%%%%%%%%%%%%%%%%%%%%%%%%%%
%%%%%%%%%%%%%%%%%%%%%%%%%%%%%%%%%%%%%%%%%%%%%%%%%%%%%%%%%%%%%%%%%%%%%%
\begin{example}\label{sec:4.6}
$R:=\dfrac{\Z}{4\Z}\ltimes~\dfrac{\Z}{2\Z}$ is an fqp total ring of quotients which is not arithmetical.
\end{example}

%%%%%%%%%%%%%%%%%%%%%%%%%%%%%%%%%%%%%%%%%%%%%%%%%%%%%%%%%%%%%%%%%%%%%%%%%%%%%%%%%%%%%%%%%%%%%%%%%%%%%%%%%%%%%%%%%%%%%
%%%%%%%%%%%%%%%%%%%%%%%%%%%%%%%%%%%%%%%%%%%%%%%%%%%%%%%%%%%%%%%%%%%%%%%%%%%%%%%%%%%%%%%%%%%%%%%%%%%%%%%%%%%%%%%%%%%%%
%%%%%%%%%%%%%%%%%%%%%%%%%%%%%%%%%%%%%%%%%%%%%%%%%%%%%%%%%%%%%%%%%%%%%%%%%%%%%%%%%%%%%%%%%%%%%%%%%%%%%%%%%%%%%%%%%%%%%
%%%%%%%%%%%%%%%%%%%%%%%%%%%%%%%%%%%%%%%%%%%%%%%%%%%%%%%%%%%%%%%%%%%%%%%%%%%%%%%%%%%%%%%%%%%%%%%%%%%%%%%%%%%%%%%%%%%%%

%%%%%%%%%%%%%%%%%%%%%%%%%%%%%%%%%%%%%%%%%%%%%%%%%%%%%%%%%%%%%%%%%%%%%%%%%%%%%%%%%%%%%%%%%%%%%%%%%%%%%%%%%%%%%%%%%%%%%
%%%%%%%%%%%%%%%%%%%%%%%%%%%%%%%%%%%%%%%%%%%%%%%%%%%%%%%%%%%%%%%%%%%%%%%%%%%%%%%%%%%%%%%%%%%%%%%%%%%%%%%%%%%%%%%%%%%%%
%%%%%%%%%%%%%%%%%%%%%%%%%%%%%%%%%%%%%%%%%%%%%%%%%%%%%%%%%%%%%%%%%%%%%%%%%%%%%%%%%%%%%%%%%%%%%%%%%%%%%%%%%%%%%%%%%%%%%
%%%%%%%%%%%%%%%%%%%%%%%%%%%%%%%%%%%%%%%%%%%%%%%%%%%%%%%%%%%%%%%%%%%%%%%%%%%%%%%%%%%%%%%%%%%%%%%%%%%%%%%%%%%%%%%%%%%%%
\end{document}